\documentclass[ECP]{ejpecp} 
\def\R{\mathbb{R}}
\def\Z{\mathbb{Z}}

\def\N{\mathbb{N}}
\def\P{\mathcal{P}}

\def\S{\mathcal{S}}
\def\F{\mathcal{F}}
\def\G{\mathcal{G}}

\def\Gm{\Gamma}
\def\l{\ell}
\def\z{\xi}
\def\g{\gamma}
\def\s{\sigma}
\def\T{\mathcal{T}}

\usepackage{mathrsfs}

\SHORTTITLE{Characterization of exchangeable sequences}

\TITLE{On the characterization of exchangeable sequences through reverse-martingale empirical distributions
    } 

\AUTHORS{%
  Martin~Bladt\footnote{Copenhagen University, Copenhagen, Denmark. Email: mab@math.ku.dk}
  \and
  Dimitry~Shaiderman\footnote{Technion – Israel Institute of Technology, Israel. Email: dima.shaiderman@gmail.com.}
 }


\AMSSUBJ{60G09; 60G48; 62G30} 

\SUBMITTED{September, 2022} 
\ACCEPTED{xx} 

\VOLUME{0}
\YEAR{2022}
\PAPERNUM{0}
\DOI{10.1214/YY-TN}

\ABSTRACT{It is a well-known fact that an exchangeable sequence has empirical distributions that form a reverse-martingale. This paper is devoted to proof of the converse statement. As a byproduct of the proof for the binary case, we introduce and discuss the notion of two-coloring exchangeability.}

\begin{document}



\section{Introduction}

The basic idea behind exchangeability is to remove the independence assumption from an independent and identically distributed (iid) sample, while preserving the same marginals and the symmetry properties that make these objects so convenient to deal with. The symmetric property is easily seen to be equivalent to the notion of invariance under permutation of indices. These distributions have been extensively studied and were first shown in 1930 by Bruno de Finetti, in a special instance in \cite{finetti}, to be equivalent to the mixing of iid samples, and later in 1937 for a more general case in \cite{finetti2}. Refer to \cite{aldous} for a good probabilistic introduction to the subject of exchangeability. This kind of symmetry also plays a role in the philosophical foundation of the Bayesian paradigm.

Let $\xi=(\xi_n)_{n\geq 1}$ be a finite or infinite random sequence defined on a probability space $(\Omega, \F, P)$ and taking values in a standard Borel space $(S,\mathcal{S})$. The sequence $\z$ is said to be \textit{exchangeable} if for every positive integer $n$
\begin{gather*} \label{exchangeable}
(\xi_{\pi(1)},\dots,\xi_{\pi(n)})\stackrel{d}{=}(\xi_{1},\dots,\xi_{n}),
\end{gather*}
for all $\pi \in \mathbb{S}_n$, where $\mathbb{S}_n$ denotes the set of permutations of $\{1,...,n\}$.  The seminal de-Finneti Theorem asserts that any infinite exchangeable sequence can be decomposed to a convex combination of iid sequences. Formally speaking, let us denote by $\mathcal{P}(\mathcal{S})$ the set of probability measures on $(S,\mathcal{S})$. Equip $\mathcal{P}(\mathcal{S})$ with the $\s$-field $\mathcal{A}$, which is generated by the mappings $\kappa \mapsto \kappa(B)$ from $\mathcal{P}(\mathcal{S}) \to [0,1]$, where $B \in \mathcal{S}$.  Then, de-Finneti's Theorem assures for any infinite  exchangeable sequence $\z=(\z_n)_{n=1}^{\infty}$ the existence of a unique probability measure $\nu$ on $(\mathcal{P}(\mathcal{S}),\mathcal{A})$, so that 
 \begin{equation*}
 P\left(\z_1 \in A_1, ... , \z_n \in A_n \right) = \int \kappa(A_1) \cdots \kappa(A_n) \, \nu(d\kappa), \	\	\	\ \forall n\in \N, 		\	\forall A_1,...,A_n \in \mathcal{S}.
\end{equation*}  
The probability measure $\nu$ is called the \textit{de-Finneti measure} of $\z$. 

 {To simplify matters, let us first consider a finite or infinite real-valued exchangeable sequence $\xi$. Denote $\mathcal{F}_n$  the $\s$-field generated by  $S_k=\sum_{i=1}^k \xi_i$, $k=n,n+1,\dots$.} We have by symmetry that
\begin{gather*}
E(\xi_{\l}\,|\,\mathcal{F}_{n+1})
\end{gather*}
is almost surely the same random variable for all $\l \leq n+1$. In particular, we get
\begin{gather*}
E(n^{-1}S_{n}|\mathcal{F}_{n+1})=E((n+1)^{-1}S_{n+1}|\mathcal{F}_{n+1})=(n+1)^{-1}S_{n+1},
\end{gather*}
thus implying that
\begin{gather*}
\left(n^{-1}S_n, \: \mathcal{F}_n\right)
\end{gather*}
is a reverse martingale. In fact, one can associate a stronger reverse-martingale property with exchangeable sequences  {which take values in a standard Borel space $(S,\mathcal{S})$ }. For any positive integer $n$ denote by $\eta_n=n^{-1}\sum_{i=1}^n \delta_{\xi_i}$  the $n$'th \textit{empirical distribution} of a random (finite or infinite) sequence $\z$, and let $\T_n$ be the $\s$-field generated by $\eta_n,\eta_{n+1},...$. 

\begin{definition}[Reverse, measure-valued martingale]
The sequence of empirical distributions $\eta:=(\eta_n)_{n\geq 1}$ is said to be a \textit{reverse, measure-valued martingale}
if for every bounded and measurable $f: S\to\mathbb{R}$ it holds that
\begin{align}\label{empiricrelat}
E(\eta_nf|\mathcal{T}_{n+1})=\eta_{n+1}f,\quad \forall n \geq 1,
\end{align}
where $\eta_n f := \int_S f(s) \eta_n (ds) = n^{-1} \sum_{i=1}^n f(\z_i)$. 
\end{definition}

 {Because for} exchangeable sequences $\z$, the prefix $(\z_1,...,\z_n)$ remains exchangeable over $\T_n$ for every $n$,  {similarly to the real-valued case,
one can easily derive the following result.} 

\begin{theorem}\label{reversemartingale.disc}
Let $\xi$ be a finite or infinite exchangeable random sequence with empirical distributions $\eta$. Then $\eta$ forms a reverse, measure-valued martingale.
\end{theorem}

 {In Theorem 2.4 of \cite{olav2}, Kallenberg states that the result of Theorem \ref{reversemartingale.disc} is not only necessary, but also sufficient for exchangeability to hold. Formally, Theorem 2.4 in \cite{olav2} states the following:} 
\begin{theorem}\label{MainThm}
Let $\z$ be a finite or infinite random sequence with empirical distributions $\eta$. Then $\z$ is exchangeable if and only if the $\eta$ form a reverse, measure-valued martingale.
\end{theorem}

 {Kallenberg's proposed proof of sufficiency turned out to be incomplete, and not mendable along the original lines. Thus, the present contribution has the principal aim of providing the first complete proof of Theorem \ref{MainThm}, using different techniques than that of Kallenberg.
}

The paper is organized as follows. In Section \ref{Sec2} we give a short proof to the case where $\z$ is a binary sequence using a Markov chain approach. This proof gives rise to the notion of two-coloring exchangeable process which we introduce and discuss in Subsection \ref{Subsec2.2}. Section \ref{Sec3} is devoted to the proof of Theorem \ref{MainThm} in the general case and builds on a discretization argument coupled with combinatorial arguments. We end the paper with Section \ref{Sec4}, in which we give a wider perspective to the main result of the paper. This section, which can be read independently right after this introduction, surveys recent developments in the theory of exchangeable sequences which are closely connected to the current work. In Appendix \ref{AppA} we discuss the original proof of Kallenberg and point to its incompleteness.

\section{Binary Case and Two-Coloring Exchangeability}\label{Sec2}
\subsection{The Binary Case}

\begin{theorem}\label{BinaryThm}
Let $\z$ be a finite or infinite sequence of binary random variables with empirical distributions $\eta_1,\eta_2,...$. Then $\z$ is exchangeable if and only if the $\eta_n$ form a reverse, measure-valued martingale.
\end{theorem}

In preparation for the proof of the above reuslt, we first establish some preliminary lemmas. Let  $Y_{n}:=\sum_{i=1}^{n} \xi_{i}, n\geq 1$, be the sequence of  partial sums of $\xi$.  

\begin{lemma}\label{Bl1}
$\xi$ is exchangeable if and only if for all $n$ and $y \in\{0, \ldots, n\}$
\begin{equation}\label{Beq1}
P\left(\xi_{1}=x_{1}, \ldots, \xi_{n}=x_{n} \mid Y_{n}= {y}\right)=\mathbf{1}_{\{ {y}\}}\left(x_{1}+\cdots+x_{n}\right)\binom{n}{ {y}}^{-1}.
\end{equation}
\end{lemma}
Further, as an elementary consequence we have

\begin{lemma}\label{Bl2}
For every $n \in \N$, the sequence $(\z_1,...\z_n)$ is exchangeable if and only if $Y=\left(Y_{1}, \ldots, Y_{n}\right)$ satisfies
\begin{equation}\label{EqTranProb}
\begin{split}
P\left(Y_{\l}=z \mid  Y_{\l+1}=y_{\l+1},...,Y_n=y_n\right)& = P\left(Y_{\l}=z \mid Y_{\l+1}=y_{\l+1}\right)\\
& = \begin{cases}\frac{y_{\l+1}}{\l+1} & \text { if }\: z=y_{\l+1}-1, \\ \frac{\l+1-y_{\l+1}}{\l+1} & \text { if } \: z=y_{\l+1}, \:\: 	\	\	\	\	\	\	\	\	\	\	\	\	\forall\l \leq n.\\ 0 & \text { otherwise }\end{cases}
\end{split}
\end{equation}
\end{lemma} Said differently, Lemma \ref{Bl2} asserts that $(\z_1,...\z_n)$ is exchangeable if and only if $Y=\left(Y_{1}, \ldots, Y_{n}\right)$ forms a reverse Markov chain with transition probabilities given by (\ref{EqTranProb}).
\begin{proof}[Proof of Lemma \emph{\ref{Bl2}}]
If $\z$ is exchangeable, this is a simple consequence of (\ref{Beq1}). If $Y$ is a reverse Markov chain with these transition probabilities, multiplying them together and using that $\xi_{i}=Y_{i}-Y_{i-1}$ yields
\begin{equation*}
\begin{split}
& P\left(\z_{1}=x_{1}, \ldots, \z_{n}=x_{n} \mid Y_{n}=y_n\right)= \\
& P\left(Y_{1}=x_{1}, \ldots, Y_{n}-Y_{n-1}=x_{n} \mid Y_{n}=y_n\right)= \\
& P\left(Y_{1}=x_{1} \mid Y_{2}=x_{2}+x_{1}\right) P\left(Y_{2}=x_{2}+x_{1} \mid Y_{3}=x_{1}+x_{2}+x_{3}\right) \cdot \\ 
& ... \cdot P\left(Y_{n-1}=x_{1}+\cdots x_{n-1} \mid Y_{n}=y_n\right)=\mathbf{1}_{\{y_n\}}\left(x_{1}+\cdots+x_{n}\right) \binom{n}{y_n}^{-1} {.}
\end{split}
\end{equation*}
By Lemma \ref{Bl1}, this implies that the sequence $\z$ is exchangeable. Note that the  Markov chain $ {\left(Y_{1}^{\prime}, Y_{2}^{\prime}, \ldots, Y_{n}^{\prime} \right)= \left(Y_{n}, Y_{n-1} \ldots, Y_{1} \right)}$ exactly describes sampling without replacement from an urn with $n$ balls, $y_n$ of which are white.
 \end{proof}
 
 \begin{proof}[Proof of Theorem \emph{\ref{BinaryThm}}]
 {As already discussed in the introduction we only need to show sufficiency. For this purpose,}  fix some $n \in \N$, and introduce for each $1 \leq \l \leq n$ the quantity
\begin{equation*}
\theta_{\l}(y_{\l+1},...,y_n)=P\left(Y_{\l}=y_{\l+1}-1 \mid Y_{\l+1}=y_{\l+1},\ldots, Y_{n}=y_{n}\right).
\end{equation*}
By the definition of $Y$ we have that
\begin{equation}\label{Beq2}
\begin{split}
E\left(Y_{\l} \mid Y_{\l+1}=y_{\l+1}, \ldots, Y_{n}=y_{n}\right) & = (y_{\l+1}-1)\theta_{\l}(y_{\l+1},...,y_n) + y_{\l+1}(1-\theta_{\l}(y_{\l+1},...,y_n))\\
& = y_{\l+1} - \theta_{\l}(y_{\l+1},...,y_n).
\end{split}
\end{equation}
On the other hand, since the $\z_{\l}$'s are binary, the empirical distributions $\eta_{\l}$ can be identified with $Y_{\l}/\l$ for each $1\leq \l \leq n$. Hence, by the reverse martingale assumption, for every such $\l$ we have
\begin{equation}\label{Beq3}
E\left(Y_{\l} \mid Y_{\l+1}=y_{\l+1},\ldots, Y_{n}=y_{n}\right) = \left(\frac{\l}{\l+1}\right)y_{\l+1}.
\end{equation}
Combining Eqs.\ (\ref{Beq2}) and (\ref{Beq3}) and yields
\begin{equation*}
\theta_{\l}(y_{\l+1},...,y_n) = y_{\l+1} - \left(\frac{\l}{\l+1}\right)y_{\l+1} = \frac{y_{\l+1}}{\l+1},  		\	\	 1\leq \l \leq n.
\end{equation*}
Hence, in view of Lemma \ref{Bl2}, we obtain that $(\z_1,...,\z_n)$ is exchangeable, as desired. 
\end{proof}

\subsection{Two-Coloring Exchangeability}\label{Subsec2.2}
Let us now introduce the notion of \emph{two-coloring exchangeable} sequences.

\begin{definition}
A finite or an infinite random sequence $\z = (\z_n)_{n \geq 1}$ taking values in a standard Borel space $(S,\mathcal{S})$ is said to be \emph{two-coloring exchangeable} if for every  {measurable} $f:S \to \{0,1\}$ the binary sequence $(f(\z_n))_{n \geq 1}$ is exchangeable.
\end{definition}
In words, a sequence is \emph{two-coloring exchangeable} if for any coloring of $S$ with two colors, the colored sequence must be exchangeable. 

\begin{proposition}\label{Obs1}
If the sequence of empirical distributions $(\eta_n)_{n \geq 1}$ of an underlying random sequence $\z = (\z_n)_{n \geq 1}$ taking values in a standard Borel space $S$ is a reverse, measure-valued martingale then $\z$ is two-coloring exchangeable. 
\end{proposition}

\begin{proof}[Proof of Proposition \emph{\ref{Obs1}}.]
Take some $f : S \to \{0,1\}$ and denote for each $n$ the $\sigma$-field $\G^f_n := \sigma (\{  \frac{1}{\l} \sum_{i=1}^{\l} f(\z_i); \l \geq n \})$. Since the empirical distributions $(\eta_n)$ of $(\z_n)$  determine the sequence of empirical distributions of the sequence $(f(\z_n))$ we have for every $n \in \N$ that
\begin{equation*}
\begin{split}
E \left( \frac{1}{n} \sum_{i=1}^n f(\z_i) \,\bigg|\, \G^f_{n+1}\right) & = E  \left( E  \left(  \eta_n f \, \bigg|\, \T_{n+1} \right) \,\bigg|\, \G^f_{n+1}\right)\\
& = E  \left( \eta_{n+1}f \,\bigg|\, \G^f_{n+1}\right)  = \frac{1}{n+1} \sum_{i=1}^{n+1} f(\z_i).
\end{split}
\end{equation*}
Since $f$ is binary-valued, the above relation establishes that the reverse martingale property holds for the empirical distribution of $(f(\z_n))_{n\geq 1}$. Thus $(f(\z_n))$ must be exchangeable by Theorem \ref{BinaryThm}. 
\end{proof}

\begin{question}\label{Q1}
Is every finite or infinite two-coloring exchangeable sequence exchangeable?
\end{question}

 {To the best of our knowledge, Question \ref{Q1} has not yet been introduced in the literature.  Although the answer to this question remains an open problem, let us now describe an interesting connection with the well-known marginal problem (e.g., Strassen \cite{Strassen}).} 

Assume that $(\z_n)_{n=1}^{\infty}$ is an infinite two-coloring exchangeable sequence taking values in the finite set $\{1,...,d\}$ for some integer $d \geq 3$. Consider for each $i \in \{1,...,d\}$ the binary function $f_i$ which assigns $1$ to $i$, and $0$ otherwise. For each such $i$, denote by  $\mu_i$ the de-Finneti measure of the exchangeable sequence $(f_i(\z_n))_{n=1}^{\infty}$. The measure $\mu_i$ is a probability measure on the unit interval $[0,1]$. For the sake of the discussion, assume that $(\z_n)_{n=1}^{\infty}$ is also exchangeable. Then, its de-Finneti's measure, denoted $\mu$, is a probability measure on the $d$-dimenstional simplex, denoted $\Delta_d$, and defined by
\begin{equation*}
\Delta_d := \{ x \in \R^d\,:\, \sum_{i=1}^d x_i =1, \, x_i \geq 0\,\,  \forall i \}.
\end{equation*}
The marginals of the measure $\mu$ (i.e., the projected measures to each coordinate of $\R^d$) are exactly the measures $\mu_i$, $i=1,...,d$. Thus, the measure $\mu$ is solution to the marginal problem on $\R^d$ with marginals $\mu_i$ and domain $\Delta_d$. Therefore,  we obtain that for finite valued infinite sequences $(\z_n)_{n=1}^{\infty}$ a necessary condition for a positive answer to Question \ref{Q1}, is that the marginal problem on $\R^d$ with marginals $\mu_i$ and domain $\Delta_d$ admits a solution. An equivalent formulation of this necessary condition can be written as follows (e.g., Theorem 3.4 in \cite{Edwards}):
\begin{gather*} \forall f_1,...,f_d \in \mathscr{C}([0,1])  \,\, \text{s.t.} \,\, \sum_{i=1}^d f_i(x_i)\geq 0 \,\,\, \forall (x_1,...x_d) \in  \Delta_d\\
\text{we have} \,\,\, \sum_{i=1}^d \int\limits_0^1 f_i d\mu_i \geq 0,
\end{gather*}
where $\mathscr{C}([0,1])$ denotes the space of continuous functions on the unit interval. 

\section{General Case}\label{Sec3}

Notice that exchangeability is defined in terms of finite-dimensional distributions. Additionaly, if the sequence of empirical distributions $(\eta_{\l})_{\l=1}^{\infty}$ of an infinite sequence $(\z_{\l})_{\l=1}^{\infty}$ forms a reverse, measure valued martingale, it follows easily from the tower property of conditional expectations that the empirical distributions $(\eta_{\l})_{\l=1}^{n}$ of the sequence  $(\z_{\l})_{\l=1}^{n}$ also form a reverse, measure valued martingale for every $n \geq 1$. Thus, it suffices to prove Theorem \ref{MainThm} for finite sequences $(\z_1,...,\z_n)$.

The first step towards proving Theorem \ref{MainThm} is to establish the following discrete version of it. 

\begin{proposition}\label{ThmA1}
Let $\z =(\z_1,...,\z_m)$ be a discrete-valued random sequence with empirical distributions $\eta_1,...,\eta_m$. Then $\z$ is exchangeable  {if} the sequence $(\eta_n)$ forms a reverse, measure-valued martingale.
\end{proposition}

\begin{proof}[Proof of Proposition \emph{\ref{ThmA1}}]
We begin with the introduction of notation and definitions. Let $\Gm$ be a countable alphabet so that $\z_n \in \Gm$ for every $n$.  {For every $n$, the set $\Gm^n$  corresponds to all words of length $n$ over the alphabet $\Gamma$. Denote $\mathcal{H}:=\bigcup_{n\leq m} \Gamma^n$ the set of all words of length at most $m$. Let $\mu$ be the probability measure induced on $\Gamma^m$ (equipped with the discrete $\s$-algebra) by the process $(\z_1, ..., \z_m)$. That is,
\begin{equation*}
\mu\left((h_1,...,h_m)\right) = P\left(\{\z_1=h_1,...,\z_m=h_m\}\right),  ~ \forall (h_1, ... , h_m) \in \Gamma^m.
\end{equation*}}
For every two words  {$h \in \Gamma^n$ and $w \in \Gamma^{\l}$}, $n+\l \leq m$, denote by $hw$ the word obtained from the concatenation of $h$ and $w$, i.e., $hw = (h_1,...,h_n,w_1,...,w_{\l}) \in  {\Gamma^{n+\l}}$. Consider the set $X \subset (\N \cup \{0\})^{\Gm}$, defined by
\begin{equation*}
X = \{(x_{\g})_{\g \in \Gm} \in (\N \cup \{0\})^{\Gm}\,:\, \sum_{\g \in \Gm} x_{\g} \leq m\}.
\end{equation*}
The set $X$ can be thought of as the set of counting measures on $\Gm$ having finite support, whose total measure is at most $m$. For each $x \in X$ let $s(x)=\{\g \in \Gm\,:\, x_{\g}>0 \}$ be the support of $x$. Denote for every $\g \in \Gm$ by $e_{\g}$ the sequence in $X$ that assigns $1$ to $\g$ and $0$ for all other letters. For each  {$x = (x_{\g})_{\g \in \Gm} \in X$} consider
\begin{equation*}
H[x] = \{h \in  {\mathcal{H}} \,:\, \g\,\, \text{appears}\,\, x_{\g}\,\,  \text{times in}\,\, h, \, \forall \g \in \Gm \}.
\end{equation*}
A simple combinatorial argument implies that for every  {$x = (x_{\g})_{\g \in \Gm} \in X$},
\begin{equation}\label{Com}
|H[x]| = \frac{\left(\sum_{\g \in s(x)} x_{\g}\right)!}{\prod_{\g \in s(x)} x_{\g}!}.
\end{equation}
Lastly, for every $n < \l < m$ positive integers let $\z_n^{\l} := (\z_n,\z_{n+1},...,\z_{\l})$. Our strategy to prove Proposition \ref{ThmA1} will be to prove by induction the following statement. For every $n \leq m$ and permutation $\pi \in \mathbb{S}_n$ it holds
\begin{equation*}\label{Suf1}
(\z_1,...,\z_n,\z_{n+1},...,\z_m) \,{\buildrel d \over =}\, (\z_{\pi(1)},...,\z_{\pi(n)},\z_{n+1},...,\z_m).
\end{equation*}
The base case $n=1$ follows immediately. Let us show the induction step $n-1 \Rightarrow n$. Consider the event $\left\lbrace \eta_n =  {n^{-1}}\sum_{\g \in \Gm} k_{\g} \delta_{\g}, \, \z_{n+1}^{m} = w \right\rbrace$ for some $w \in  {\Gamma^{m-n}}$ and some non-negative integers $(k_{\g})_{\g \in \Gm}$, and assume that it has positive probability. The sequence $x:= (k_{\g})_{\g \in \Gm}$ lies in $X$. By the induction step we have
\begin{equation}\label{IndSt}
\mu(v\g w) = \mu(v'\g w)\, 	\	\	\ \, \forall \g \in s(x), \, \forall v,v' \in H[x-e_{\g}].
\end{equation} 
Therefore, if we can show that for any two $\g^* \neq \g^{**} \in s(x)$ and any two histories $h$ and $h'$ in $H[x-e_{\g^*}]$ and $H[x-e_{\g^{**}}]$, respectively we have that $\mu(h\g^* w) = \mu(h'\g^{**} w)$, we will obtain that for every $\pi \in \mathbb{S}_n$ we have
\begin{equation*}
(\z_1,...,\z_n,\z_{n+1},...\z_{m}) \,{\buildrel d \over =}\, (\z_{\pi(1)},...,\z_{\pi(n)},\z_{n+1},...,\z_{m}),
\end{equation*}
as desired. To show the required relation, fix $\g^* \neq \g^{**} \in s(x)$ and two histories $h$ and $h'$ in $H[x-e_{\g^*}]$ and $H[x-e_{\g^{**}}]$, respectively. We need the following two claims:
\begin{claim}\label{O1}
$\s(\eta_n, \z_{n+1}^{m}) = \T_n.$
\end{claim}
\begin{proof}[Proof of Claim \emph{\ref{O1}}]
Follows from the fact that  $\l \eta_{\l} - (\l-1)\eta_{\l-1} = \delta_{\z_{\l}}$ for every $\l \in \N$.  
\end{proof}
\begin{claim}\label{O2}
For every bounded function $f: \Gm \to \R$ we have $E (f(\z_n)\,|\, \T_n) = \eta_n f$. 
\end{claim}
\begin{proof}[Proof of Claim \emph{\ref{O2}}]
Follows from the relation $f(\z_n) = n\eta_n f - (n-1) \eta_{n-1} f$ and the reverse-martingale property of $(\eta_n)$. 
\end{proof}
By Claims \ref{O1} and \ref{O2} we infer that for every bounded $f: \Gm \to \R$ we have $$E (f(\z_n)\,|\,\eta_n, \z_{n+1}^{m}) = \eta_n f.$$ Consider the function $f(\g)=I\{\g= \g^*\}$.  {On the one hand for each $\omega \in \lbrace \eta_n = n^{-1}\sum_{\g \in \Gm} k_{\g} \delta_{\g}, \, \z_{n+1}^{m} = w \rbrace$,  the reverse martingale property implies that
\begin{equation}\label{EqM1}
\begin{split}
E \left(f(\z_n)\,\Bigg|\, \left\lbrace \eta_n =n^{-1} \sum_{\g \in \Gm} k_{\g} \delta_{\g}, \, \z_{n+1}^{m} = w \right\rbrace \right) & = E \left(f(\z_n)\,\Bigg|\, \eta_n, \z_{n+1}^{m} \right)(\omega) \\
& = \frac{1}{n} \sum_{i=1}^n I\{\z_i= \g^*\}(\omega)
= \frac{k_{\g^*}}{n}.
\end{split}
\end{equation}}
On the other hand,  {by the earlier definition $x := (k_{\g})_{\g \in \Gamma}$,  utilizing  (\ref{IndSt}) and (\ref{Com}) we obtain}
\begin{multline*}
E \left(f(\z_n)\,\Bigg|\, \left\lbrace \eta_n =  {n^{-1}} \sum_{\g \in \Gm} k_{\g} \delta_{\g}, \, \z_{n+1}^{m} = w \right\rbrace \right)\\
 = \frac{\sum_{ v \in H[x-e_{\g^*}]} \mu(v \g^* w)}{P\left(\left\lbrace \eta_n =  {n^{-1}}\sum_{\g \in \Gm} k_{\g} \delta_{\g}, \, \z_{n+1}^{m} = w \right\rbrace \right)}\\
= \frac{(n-1)!}{(k_{\g^*}-1)! \prod_{\g \in s(x)/\{\g^*\}} k_{\g}! } \cdot \frac{\mu(h\g^*w)}{P\left(\left\lbrace \eta_n =  {n^{-1}}\sum_{\g \in \Gm} k_{\g} \delta_{\g}, \, \z_{n+1}^{m} = w \right\rbrace \right)},
\end{multline*}
which in turn with (\ref{EqM1}) implies 
\begin{equation*}
\mu(h\g^*w)  = \frac{\prod_{\g \in s(x)} k_{\g}! }{n!}P\left(\left\lbrace \eta_n =  {n^{-1}}\sum_{\g \in \Gm} k_{\g} \delta_{\g}, \, \z_{n+1}^{m} = w \right\rbrace \right).
\end{equation*}
By performing a parallel analysis with $f(\g) = I\{\g = \g^{**}\}$ we obtain by the same arguments that
\begin{equation*}
\mu(h'\g^{**}w)  = \frac{\prod_{\g \in s(x)} k_{\g}! }{n!}P\left(\left\lbrace \eta_n =  {n^{-1}} \sum_{\g \in \Gm} k_{\g} \delta_{\g}, \, \z_{n+1}^{m} = w \right\rbrace \right)
\end{equation*}
as well. Therefore, we have shown that $\mu(h\g^*w) = \mu(h'\g^{**}w)$ as required. 
\end{proof}

We may now proceed with the proof of Theorem \ref{MainThm}.

\begin{proof}[Proof of Theorem \emph{\ref{MainThm}}]
 {From the introductory remarks, we may concentrate on proving sufficiency. Thus assume that $\eta$ forms a reverse, measure-valued martingale.} By the Borel Isomorphism Theorem (cf. p.\ 49-50 in Aldous \cite{aldous}) it suffices to consider the case $S = \mathbb{R}$, as the extension of the result to any standard Borel space follows immediately. Let $\mu$ be the probability measure induced by $(\z_1,...,\z_m)$ on $(\R^{m},\mathcal{B}(\R^{m}))$. For each $d \in \N$, consider the step function $g^d:\R \to \R$ given by $$g^d(t)= \sum_{\l \in \Z} \frac{\l}{2^d}\, I\bigg\{t \in \Big( \frac{\l-1}{2^d} , \frac{\l}{2^d} \Big]\bigg\}.$$ Define the discrete-valued sequence of random variables $Y^d = (Y_1^d,...,Y_m^d)$ by $Y^d_n := g^d(\z_n)$. Let $(\eta^d_1,...,\eta_m^d)$ be the sequence of empirical measures of the sequence $Y^d$, and denote $\T^d_n = \s(\{\eta^d_i\,:\, i \geq n\})$. Since the sequence $(\eta^d_n)$ is adapted to the sequence $(\eta_n)$, and the latter is a reverse, measure-valued martingale, we have for every bounded and measurable $f : \R \to \R$ and any positive integer $n < m$ that
\begin{equation*}
\begin{split}
E  \left( \eta^d_n f\,\bigg|\, \T^d_{n+1}\right) &= E \left( E \left(\eta_n (f\circ g^d)\, \bigg|\, \T_{n+1} \right) \,\bigg|\, \T^d_{n+1}\right)\\
& =  E  \left( \eta_{n+1} (f\circ g^d) \,\bigg|\, \T^d_{n+1}\right)\\
& = E  \left( \eta^d_{n+1} f \,\bigg|\, \T^d_{n+1}\right) = \eta^d_{n+1} f.
\end{split}
\end{equation*}
Thus, the sequence $(\eta^d_n)_{n\geq 1}$ also forms a reverse, measure-valued martingale. Hence, by  {Proposition} \ref{ThmA1} we obtain that $Y^d$ is exchangeable for any $d \in \N$.

 The latter yields that the sequence $(\z_1,...,\z_m)$ is exchangeable on dyadic cubes\footnote{A \emph{dyadic cube} $Q \subset \mathbb{R}^m$ is a set of the form $\displaystyle{ \frac{z}{2^d} + \Big( 0,\frac{1}{2^d} \Big]^m}$ for some  $d \in \mathbb{N}$, $z \in \mathbb{Z}^m$.}. Indeed for every $d \in \mathbb{N}$ and any permutation $\pi \in \mathbb{S}_m$ we get
\begin{eqnarray} \label{Dyadic EX}
\mu \bigg(\Big( \frac{\l_1-1}{2^d} , \frac{\l_1}{2^d} \Big] ,...,\Big( \frac{\l_m-1}{2^d} , \frac{\l_m}{2^d} \Big] \bigg) & = &P \bigg(Y_1^d = \frac{\l_1}{2^d},...,Y_m^d = \frac{\l_m}{2^d} \bigg) \nonumber	\\
& = & P \bigg(Y_1^d  = \frac{\l_{\pi(1)}}{2^d},...,Y_m^d = \frac{\l_{\pi(m)}}{2^d} \bigg)\\
& = & \mu \bigg( \Big( \frac{\l_{\pi(1)}-1}{2^d} , \frac{\l_{\pi(1)}}{2^d} \Big] ,...,\Big( \frac{\l_{\pi(m)}-1}{2^d} , \frac{\l_{\pi(m)}}{2^d} \Big] \bigg). \nonumber
\end{eqnarray}
Since $\mu$ is a regular measure, and as any open set in $\mathbb{R}^m$ can be decomposed into a disjoint union of dyadic cubes, we have the following identity:
\begin{equation} \label{Outer rep}
\mu(B) = \inf \Bigg \lbrace \sum\limits_{k=1}^{\infty} \mu(Q_k); \, Q_k \subset \mathbb{R}^m \, \text{are disjoint dyadic cubes and}\, B \subset \bigcup_k Q_k  \Bigg \rbrace.
\end{equation}
Take some $\pi \in \mathbb{S}_m$ and define the map $\pi: \mathbb{R}^m \to \mathbb{R}^m$ by
$\pi(x_1,...,x_m) := (x_{\pi(1)},...,x_{\pi(m)})$. By Eq.\ (\ref{Dyadic EX}) we have that for each dyadic cube $Q \subset \mathbb{R}^n$,
\begin{equation}\label{Dyadic EX Short}
\mu(Q) = \mu(\pi(Q)).
\end{equation}
Take some events $A_1,...,A_m \in \mathcal{B}(\R)$ and consider $A_1\times ...\times A_m$. Since $\pi$ is a bijection of $\R^m$, if $A_1\times ...\times A_m \subset \bigcup_k Q_k$ for some disjoint dyadic cubes $Q_k$, then $(\pi(Q_k))_k$ are also disjoint dyadic cubes and $$A_{\pi(1)}\times ...\times A_{\pi(m)}  = \pi(A_1\times ...\times A_m) \subset \bigcup_k \pi(Q_k).$$
The latter observation, together with Eqs.\ (\ref{Outer rep}) and (\ref{Dyadic EX Short}) yields that
\begin{equation*}
\mu\left(A_1\times...\times A_m \right) = \mu\left(A_{\pi(1)}\times...\times A_{\pi(m)} \right).
\end{equation*}
As $\pi \in \mathbb{S}_m$ was arbitrary we have shown that $(\z_1,...,\z_m)$ is exchangeable.
\end{proof}

We end the paper with a perspective, which in our view, places the main result of the current paper in a broader context, involving random probability measures, martingales, and exchangeability.

\section{Exchangeability, Random Probability Measures and Martingales}\label{Sec4}
Consider a probability space $(\Omega,\F,P)$ and a standard Borel space $(S,\S)$. Denote by $\P(\S)$ the space of all Borel probability measures on $(S,\S)$. Equip $\P(\S)$ with the $\s$-field $\mathcal{A}$ generated by the functions $\kappa \to \kappa(B)$, $\kappa \in \P(\S)$,  where $B$ ranges over $\S$. A random probability measure (r.p.m) $\nu$ is a measurable mapping  from $(\Omega,\F)$ to $(\P(\S),\mathcal{A})$. Thus, for each $\omega \in \Omega$, $\nu(\omega)$ is a Borel probability measure on $(S,\S)$. Throughout the paper we studied the r.p.m.'s $(\eta_n)_{n\geq 1}$ corresponding to the empirical measures induced by a sequence of random variables $(\z_n)_{n\geq 1}$ defined on  $(\Omega,\F,P)$, and taking values in $S$. 

Another fundamental family of r.p.m.'s associated with the sequence $(\z_n)_{n\geq 1}$ is that of the predictive distributions $(p_n)_{n\geq 1}$ of  $(\z_n)_{n\geq 1}$, given by
\begin{equation*}
(p_n(\omega))(B):= P \left(\z_{n} \in B\,|\, \z_1,...,\z_{n-1} \right) (\omega), 	\	\	 	\	\forall B \in \S, 
\end{equation*}
where for each $n$, we may take $P \left(\z_{n} \in \cdot \,|\, \z_1,...,\z_{n-1} \right)$ to be a regular conditional probability. The pioneering work of Kallenberg \cite{kallenberg1988spreading} studied the case where the sequence of r.p.m.'s $(p_n)_{n\geq 1}$ satisfies the measure-valued martingale condition. Formally, it means that for every bounded and measurable $f:S \to \R$, the sequence of random variables $(p_n f)(\omega) := \int f(x)[p_n(\omega)](dx)$, $\omega \in \Omega$, forms a martingale. It follows easily that whenever $(\z_n)_{n\geq 1}$ is exchangeable, $(p_n)_{n\geq 1}$ forms a measure-valued martingale. However, unlike in the case where we impose the \textbf{reverse martingale} condition on $(\eta_n)_{n\geq 1}$, the mere martingale assumption on $(p_n)_{n\geq 1}$ need not be sufficient for exchangeability (see examples 1.1, 1.2, and 1.3 in Berti, Pratelli, and Rigo \cite{Berti}). Nevertheless, Kallenberg was able to show the following criterion for exchangeability:
\begin{theorem}[Kallenberg, \cite{kallenberg1988spreading}]  
The infinite sequence  $(\z_n)_{n=1}^{\infty}$ is exchangeable if and only if it is stationary and $(p_n)_{n=1}^{\infty}$ forms a measure-valued martingale.
\end{theorem} 
Recently, the stationarity assumption played a key role in yet another characterization of infinite exchangeable sequences in terms of the r.p.m.'s $(p_n)_{n=1}^{\infty}$. Lehrer and Shaiderman (see Theorem 6 in \cite{lehrer2020exchangeable}) showed that exchangeability of infinite sequences is equivalent to stationarity and the strong permutation invariant posteriors (strong PIP) property, where the later is given by the condition that
\begin{equation*}
P \left(\z_{n} \in \cdot \,|\, \z_1 \in A_1,...,\z_{n-1}\in A_{n-1} \right) = P \left(\z_{n} \in \cdot \,|\, \z_1 \in A_{\pi(1)},...,\z_{n-1}\in A_{\pi(n-1)} \right),
\end{equation*}
for every $n \in \N$, $A_1,...,A_{n-1} \in \S$, and any permutation $\pi$ of $\{1,...,n-1\}$,  provided that $P\left(\lbrace \z_1 \in A_{\pi(1)},...,\z_{n-1}\in A_{\pi(n-1)} \rbrace \right)>0$ for any such $\pi$.

The martingale property of $(p_n)_{n=1}^{\infty}$ turned out to be important in its own right. Berti, Pratelli, and Rigo \cite{Berti} showed that infinite random sequences $(\z_n)_{n=1}^{\infty}$  for whom $(p_n)$ forms a measure-valued martingale (also known as the c.i.d. condition) posses a rich probabilistic structure, admitting strong laws and central limit type theorems. In particular, they generalized such known results on infinite exchangeable sequence to a wider class, that  {of} c.i.d.\ sequences.

A recent part of the literature on exchangeable sequences is devoted to the proximity of the r.p.m.'s $\eta_n$ and $p_n$ for large values of $n$ (see \cite{Berti17} for a survey of methodologies and new results). To give an example of such proximity result, assume that $S=\R^d$ for some $d\in \N$, and denote by $\mathbb{B}_d$ the set of closed euclidean balls in $\R^d$. Berti, Pratelli, and Rigo \cite{Berti18} showed that (e.g., Corollary 3 in \cite{Berti18}) if $(\z_n)_{n\geq 1}$ is exchangeable then
\begin{equation*}
r_n \sup_{B \in \mathbb{B}_d} |\eta_n(B) - p_n(B) | \to 0\, 	\	\ \text{a.s.}  \ 	\	\	\text{whenever} 	\	\		\ r_n \sqrt{\frac{\log \log n}{n}} \to 0.
\end{equation*}

\begin{acks}
The authors wish to thank Steffen L. Lauritzen and Ehud Lehrer for their contribution to this work. Ehud is also acknowledged for highlighting the connection between the two-coloring exchangeability property and the marginal problem discussed in Subsection \ref{Subsec2.2}.  {Finally, the authors wish to thank an anonymous reviewer for their careful reading and comments which significantly improved the content of the paper.} 
\end{acks}

\appendix 

\section{Kallenberg's Argument}\label{AppA}
As mentioned in the introduction, a proof of the Theorem \ref{MainThm} was given in \cite{olav2}, using a different proof technique. However, the proof appears to not be complete. To be more specific, consider a finite random sequence $(\z_1,...,\z_n)$ for whom $(\eta_{\l},\T_{\l})_{\l \leq n}$ forms a reverse, measure valued martingale. Then, by (\ref{empiricrelat}) for bounded, measurable $f$ on $S$, 
\begin{gather}\label{reversecrux}
E(f(\xi_{k})|\mathcal{T}_k)=k\eta_k f-(k-1)E(\eta_{k-1}f|\mathcal{T}_k)=\eta_k f.
\end{gather}
Define for each $1\leq k \leq n$,
\begin{gather*}
\zeta_k:=\xi_{n-k+1},\:\:\: \beta_k:=\sum_{j\leq k}\delta_{\zeta_j}=n\eta_n-(n-k)\eta_{n-k},\:\:\: \mathcal{F}_k:=\mathcal{T}_{n-k}.
\end{gather*}
The sequence $\zeta=(\zeta_1,...\zeta_n)$ is $\mathcal{F}$-adapted and by (\ref{reversecrux}) we have for every $1\leq k<n$ that,
\begin{gather*}\label{marginalurncruxrev}
E(f(\zeta_{k+1})|\beta_n,\mathcal{F}_k)=E(f(\xi_{n-k})|\mathcal{T}_{n-k})=\eta_{n-k}f=(n-k)^{-1}(\beta_n-\beta_k)f,
\end{gather*}
where the first equality follows from $\beta_n = n\eta_n \in \T_{n-k}$ for every $k<n$. The latter gives us that $\zeta$ satisfies the requirements of the definition of a \textit{conditional urn-sequence} as defined in p.\ 71 in \cite{olav2}. Kallenberg then deduces his reverse martingale characterization by  utilizing Proposition 2.2 in \cite{olav2}, which argues that every adapted conditional urn-sequence must be exchangeable.

The problem thus lies in the part of the proof of Proposition 2.2, where it is argued that any finite, $\mathcal{F}$-adapted conditional urn sequence $\z=(\z_1,...,\z_n)$ is also exchangeable. To avoid confusion, in the current discussion  $\mathcal{F}$ is a general filtration, and not the one defined in terms of the $\T_k$'s before. Thus, we know that for any $k<n$ it holds that  
\begin{gather}\label{conditionalurn}
P(\xi_{k+1}\in\cdot|\beta_n,\mathcal{F}_k)=\frac{\beta_n-\beta_k}{n-k}, \:\:\:k<n,
\end{gather}
where $\beta_{\l}:=\sum_{j\leq \l}\delta_{\xi_j}$ for every $1\leq \l \leq n$.

To show that $\z=(\z_1,...,\z_n)$ is exchangeable it suffices to argue that for any $1\leq k < n$ the expression
\begin{equation}\label{TargetedSymmetry}
E(f_{k+1}(\xi_{k+1})\cdots f_n(\xi_{n})|\mathcal{F}_k,\beta_n)
\end{equation}
is symmetric in the non-negative measurable functions $f_{k+1},...,f_{n}$. For the latter, Kallenberg makes use of \emph{factorial measures}, defined on p.30 in \cite{olav2}. This measure, associates with every finite counting measure $\nu = \sum_{i=1}^{r} \delta_{s_i}$, $s_i \in S$, a measure on $S^r$, denoted $\nu^{(r)}$ and called the factorial measure of $\nu$, which is defined by
\begin{equation*}
\nu^{(r)}(B) = \sum_{\pi \in \mathbb{S}_r} I((s_{\pi(1)},...,s_{\pi(r)}) \in B), \:\: B \in \mathcal{S}^r.
\end{equation*}
In particular, the factorial measure is a symmetric measure on $S^r$, in the sense that:
\begin{equation}
\nu^{(r)}(A_1 \times ... \times A_r) = \nu^{(r)}(A_{\pi(1)}\times ... \times A_{\pi(r)}),	\		\	\forall \pi \in \mathbb{S}_r, 	\	\	\  A_1,...,A_r \in \mathcal{S}.
\end{equation}

Using the notion of factorial measures, Kallenberg then argues that the symmetry in the $f_i$'s in Eq.\ ({\ref{TargetedSymmetry}}) follows from the following calculation:
\begin{align*}\label{equivdefcalc}
E(f_{k+1}(\xi_{k+1})\cdots f_n(\xi_{n})|\mathcal{F}_k,\beta_n)&=E(f_{k+1}(\xi_{k+1})\cdots f_{n-1}(\xi_{n-1})E(f_n(\xi_{n})|\mathcal{F}_{n-1},\beta_n)|\mathcal{F}_k,\beta_n) \nonumber\nonumber\\
&= \frac{\beta_n-\beta_{n-1}}{n-n+1}f_n\: E(f_{k+1}(\xi_{k+1})\cdots f_{n-1}(\xi_{n-1})|\mathcal{F}_{k+1},\beta_n) \nonumber \\
&=\cdots=\prod_{j=k+1}^n\frac{(\beta_n-\beta_{j-1})}{n-j+1} f_{j}\\
&=\frac{(\beta_n-\beta_k)^{(n-k)}}{(n-k)!}\bigotimes_{j=k+1}^n f_j.\nonumber
\end{align*}
However, we believe that there are two inaccuracies in that calculation. First,  in the second equality the measurability of $\beta_{n-1}$ with respect to $\mathcal{F}_k,\beta_n$ was implied, which in general can not be assumed. Secondly, the last equality above need not hold in general (try with $k=0$, $n=2$). To the best of our knowledge those two inaccuracies cannot be easily mended. 

It is also worth mentioning that the incompleteness of the proof of Proposition 2.2 in \cite{olav2} just discussed, has repercussions on the proof of Theorem 2.3 in \cite{olav2} and also of Theorem 2.2 in \cite{olav2}. The proof of the latter theorem is very easily mendable, along the original lines, and hence omitted in the present work.

\end{document}